\documentclass[12pt,reqno]{amsart}

\usepackage[english]{babel}
\usepackage{comment}
\usepackage{enumerate}
\usepackage{palatino}

\usepackage{amsmath}
\usepackage{amsthm}
\usepackage{amssymb}
\usepackage{psfrag}
\numberwithin{equation}{section}

  \newcommand{\N}{\mathbb{N}}         
  \newcommand{\R}{\mathbb{R}}         

  \newcommand{\PP}{\mathcal{P}}

\newcommand{\iii}{\mathtt{i}}
\newcommand{\jjj}{\mathtt{j}}

  \newcommand{\e}{\varepsilon}

  \newcommand{\lam}{\lambda}

  \newcommand{\emdef}{\emph}

  \newtheorem{thm}{Theorem}[section]
  \newtheorem{lemma}[thm]{Lemma}
  \newtheorem{prop}[thm]{Proposition}
  \newtheorem{cor}[thm]{Corollary}

  \theoremstyle{remark}

\begin{document}

\title[The exceptional set for absolute continuity]{On the exceptional set for absolute continuity of Bernoulli convolutions}
\author{Pablo Shmerkin}


\thanks{The author was supported by a Leverhulme Early Career Fellowship}
\address{Department of Mathematics. University of Surrey,
Guildford, GU2 7XH. UK.}
\email{p.shmerkin@surrey.ac.uk}

\begin{abstract}
We prove that the set of exceptional $\lambda\in (1/2,1)$ such that the associated Bernoulli convolution is singular has zero Hausdorff dimension, and likewise for biased Bernoulli convolutions, with the exceptional set independent of the bias. This improves previous results by Erd\"{o}s, Kahane, Solomyak, Peres and Schlag, and Hochman. A theorem of this kind is also obtained for convolutions of homogeneous self-similar measures. The proofs are very short, and rely on old and new results on the dimensions of self-similar measures and their convolutions, and the decay of their Fourier transform.
\end{abstract}

\maketitle

\section{Introduction and statement of main results}

\subsection{Introduction}

We start by describing the contribution of this article in general terms; precise definitions and results are postponed to the following sections.

Although the Hausdorff dimension of self-similar and many other fractal measures on the real line can be calculated explicitly when the construction does not involve complicated overlaps (say, under the open set condition), the situation is dramatically more difficult in the overlapping case. In most problems of interest, there is a number $s$ which represents the ``symbolic'' dimension, or what the dimension would be in the absence of overlaps. In the self-similar case, this is simply the similarity dimension. There are two main regimes in the study of measures of overlapping construction:
\begin{itemize}
\item The sub-critical regime, $s<1$. In this case, one expects that ``typically'' the Hausdorff dimension of the measure equals $s$.
\item The super-critical regime, $s>1$. Here, the expectation is that ``typically'' the measure is absolutely continuous.
\end{itemize}
There are many results to the effect that for Lebesgue a.e. parameter in a given parametrized family, the expected behavior holds (in both regimes), and in many cases of interest it is was shown by Peres and Schlag \cite{PeresSchlag00} that the dimension of the exceptional set is less than full. However, it is believed that in most of the important cases, the set of exceptions is in fact countable, so a substantial gap remained. All of this applies, in particular, to the problem of absolute continuity for Bernoulli convolutions.

In the last few years, significant progress was achieved in verifying that the exceptional set in the sub-critical regime is indeed very small (countable or zero-dimensional) for many important fractal families, including convolutions of central Cantor measures \cite{NPS12}, convolutions of $\times p$-invariant measures \cite{HochmanShmerkin12}, projections of self-similar measures \cite{HochmanShmerkin12,Hochman13}, Bernoulli convolutions and other parametrized families of self-similar measures \cite{Hochman13}. However, the proofs of these results do not yield any progress on the absolute continuity of the measures in question in the super-critical regime.

In this article, we show that, for Bernoulli convolutions and convolutions of homogeneous self-similar measures, the exceptional parameters also for the absolute continuity problem in the super-critical regime are contained in a set of zero Hausdorff dimension. Other natural classes of examples will be treated in a forthcoming paper \cite{ShmerkinSolomyak13}.

Perhaps surprisingly, the proofs are very short, and rely on three main elements:
\begin{enumerate}
\item The fact that the Fourier transform of homogeneous self-similar measures has at least power decay off a zero dimensional set of contraction ratios. This essentially goes back to Erd\"{o}s \cite{Erdos40} and Kahane \cite{Kahane71}.
\item The recent dimension results alluded to earlier, and in particular the fact that in the super-critical region, the measures are already known to have full dimension off a zero-dimensional set. We emphasize that in each of our main examples (Theorems \ref{thm:homogeneous} and \ref{thm:convolution-ssm}), the corresponding dimension result was obtained through a completely different mechanism.
\item We employ the fact that the measures of interest have a convolution structure, and the previous steps, to express them as the convolution of a measure of full Hausdorff dimension, and another measure whose Fourier transform has power decay. Employing the convolution structure of Bernoulli convolutions to upgrade their smoothness is an old idea, again going back to Erd\"{o}s \cite{Erdos40} and also used in e.g. \cite{Solomyak95,PeresSolomyak96}. A novel feature of our decomposition is that each of the measures that we convolve plays a completely different r\^{o}le, and we rely on different results to show that each of them has the desired behavior off a zero dimensional parameter set.
\end{enumerate}

\subsection{Self-similar sets and measures}

We denote by $\mathcal{P}$ the set of Borel probability measures on $\R$.  If $f:\R\to \R$ is any function and $\mu\in\mathcal{P}$, we denote the push-forward of $\mu$ via $f$ by $f\mu$, i.e. $f\mu(B)=\mu(f^{-1}(B))$ for all Borel sets $B\subset\R$.

Recall that an \emdef{iterated function system} (IFS) is a finite family $\mathcal{F}=(f_1,\ldots, f_m)$ of strict contractions on some complete metric space $X$. It is well known that there exists a unique nonempty compact set $A=A(\mathcal{F})$ such that $A=\bigcup_{i=1}^m f_i(A)$. The IFS $\mathcal{F}$ satisfies the \emdef{strong separation condition} (SSC) if the pieces $f_i(A)$ are mutually disjoint.

We denote the open simplex in $\R^m$ by $\mathbb{P}_m$, i.e. $\mathbb{P}_m=\{ (p_1,\ldots,p_m): p_i> 0, \sum_{i=1}^m p_i=1\}$. We think of elements of the simplex as probability vectors. If  $p\in\mathbb{P}_m$, then there exists a unique measure $\mu=\mu(\mathcal{F},p)\in\mathcal{P}$ such that $\mu=\sum_{i=1}^m p_i\cdot f_i\mu$. The support of $\mu(\mathcal{F},p)$ is $A(\mathcal{F})$.

In this article we always assume that $X=\R$, and (with one exception) that the $f_i$ are similarities, in which case $A$ is known as a \emdef{self-similar set} and $\mu$ as a \emdef{self-similar measure}.

Recall that the \emdef{similarity dimension} $s=s(\mathcal{F})$ is the unique positive number $s$ such that $\sum_{i=1}^m r_i^s=1$, where $r_i$ is the contraction ratio of $f_i$. Further, if a probability vector $p\in\mathbb{P}_m$ is also given, the similarity dimension $s=s(\mathcal{F},p)$ is
\[
s= \frac{\sum_{i=1}^m p_i \log p_i}{\sum_{i=1}^m p_i \log r_i}.
\]
We denote Hausdorff dimension of sets by $\dim_H$ and \emdef{lower Hausdorff dimension} of  measures by $\dim$; this is defined as
\[
\dim\mu=\inf\{\dim_H(B):\mu(B)>0\}.
\]
It is well known that $\dim_H A(\mathcal{F})\le s(\mathcal{F})$, and likewise $\dim\mu(\mathcal{F},p)\le s(\mathcal{F},p)$, with equality under the SSC, see \cite[Corollary 5.2.3 and Theorem 5.2.5]{Edgar98}. In particular, if $s(\mathcal{F})<1$, then $A(\mathcal{F})$ has zero Lebesgue measure and, likewise, if $s(\mathcal{F},p)<1$, then $\mu(\mathcal{F},p)$ is singular with respect to Lebesgue measure (it is well known that self-similar measures are either absolutely continuous or purely singular). We follow a usual abuse of notation and speak of the similarity dimension of $A$ or $\mu$ when it is clear from context to which generating IFS we are referring.

\subsection{Bernoulli convolutions}

The Bernoulli convolution $\nu_\lambda^p$, corresponding to a contraction ratio $\lambda\in (0,1)$ and a weight $p\in (0,1)$, is the distribution of the random sum $\sum_{n=0}^\infty \pm \lambda^n$, where $P(+)=p$, $P(-)=1-p$ and all the choices are independent. In the unbiased case $p=1/2$ we simply write $\nu_\lambda=\nu_\lambda^{1/2}$. Alternatively, $\nu_\lambda^p$ is the self-similar measure corresponding to the IFS $(\lambda x-1,\lambda x+1)$ with weights $(p,1-p)$.  The survey article \cite{PSS00} provides an excellent overview of the major problems and results on Bernoulli convolutions up to the year 2000.

When $\lambda\in (0,1/2)$, the support of $\nu_\lambda$ is a self-similar Cantor set of Hausdorff dimension $\log 2/|\log \lam|<1$. For the critical value $\lambda=1/2$, $\nu_{\lambda}$ is normalized Lebesgue measure on its support.  The study of Bernoulli convolutions in the super-critical region $\lambda\in (1/2,1)$ was pioneered by Erd\"{o}s \cite{Erdos39, Erdos40}, who proved that if $1/\lambda$ is a Pisot number (i.e. an algebraic integer larger than $1$, all of whose algebraic conjugates lie in the open unit disk), then $\nu_\lambda$ is singular, and indeed $\widehat{\nu}_\lambda(\xi)\nrightarrow 0$ as $|\xi|\to\infty$ where here, and throughout the article, $\widehat{\mu}$ denotes the Fourier transform of a Borel finite measure $\mu$ on $\R$, defined as
\[
\widehat{\mu}(\xi) = \int e^{i \pi \xi x }d\mu(x).
\]
(We choose this somewhat non-standard normalization for practical reasons to become apparent later.) It is still an open problem to determine whether there are any other values of $\lambda\in (1/2,1)$ for which $\nu_\lambda$ is singular. We quickly review the history of the progress on this problem. Erd\"{o}s \cite{Erdos40} proved the existence of some $\delta>0$ such that $\nu_\lambda$ is absolutely continuous for almost all $\lambda \in (1-\delta,1)$. Kahane \cite{Kahane71} showed that in fact
\[
\dim_H(\{\lambda\in (1-\delta,1): \nu_\lambda \text{ is not absolutely continuous }\}) \to 0 \,\text{ as }\, \delta\to 0.
\]
See \cite[Section 6]{PSS00} for an exposition of the Erd\"{o}s-Kahane argument. While these results hold only for $\lambda$ near $1$, Solomyak obtained a major breakthrough in \cite{Solomyak95}, proving that $\nu_\lambda$ is absolutely continuous with an $L^2$ density for almost all $\lambda \in (1/2,1)$. A simpler proof was shortly after obtained by Peres and Solomyak \cite{PeresSolomyak96}. In \cite{PeresSchlag00}, Peres and Schlag improved Solomyak's result in two directions: they proved that for any $a\in (1/2,1)$, $\nu_\lambda$ has fractional derivatives in $L^2$ for $\lambda\in (a,1)$ outside of a set of Hausdorff dimension $d(a)<1$.

While the above discussion is for the unbiased case $p=1/2$, most of the results have extensions to the biased case. The similarity dimension of $\nu_\lam^p$ is
\[
s(\lam,p) = \frac{h(p)}{-\log\lambda} = \frac{-(p\log p+(1-p)\log(1-p))}{-\log\lambda}.
\]
It follows that if $s(\lam,p)<1$, or equivalently if $\lambda<p^p (1-p)^{1-p}$, then $\nu_\lam^p$ is always singular. Peres and Solomyak \cite{PeresSolomyak96,PeresSolomyak98} showed that if $p\in [1/3,2/3]$, then $\nu_\lambda^p$ is absolutely continuous for almost all $\lambda \in (p^p (1-p)^{1-p},1)$. Lindenstrauss, Peres and Schlag \cite{LPS02} improved on this by showing that the exceptional set of $\lambda$ can be taken to be independent of $p$. T\'{o}th \cite{Toth08} proved that even for $p$ outside of $[1/3,2/3]$, $\nu_\lambda^p$ is absolutely continuous for almost all $\lambda\in (1-\delta,1)$ for some explicit but non-optimal $\delta=\delta(p)$.

Very recently, Hochman \cite{Hochman13} made major progress in the dimension theory of self-similar measures, and in particular proved that for all $\lambda$ outside of a set of Hausdorff (and even packing) dimension $0$, and all $p\in (0,1)$,
\[
\dim\nu_\lambda=\min(s(\lam,p),1).
\]
See Theorem \ref{thm:zero-dim-excpetion-for-dim} below for a more general statement. While Hochman's method does not directly give any result on absolute continuity (see \cite[Section 1.5]{Hochman13}), in this paper we use his result to obtain a short proof of the following:
\begin{thm} \label{thm:BC-biased}
There exists a set $E\subset (1/2,1)$ of Hausdorff dimension $0$, such that $\nu_\lambda$ is absolutely continuous for all $p,\lambda$ such that $s(\lam,p)>1$ and $\lambda\in (1/2,1)\setminus E$.
\end{thm}
In fact, we prove a more general statement:
\begin{thm} \label{thm:homogeneous}
Let $a_1,\ldots, a_m$ be distinct fixed numbers, and for $\lam\in (0,1)$, let $\mathcal{F}_\lambda = (\lam x+a_1,\ldots, \lam x +a_m)$. There exists a set $E$ of zero Hausdorff dimension, such that if $\lam\in(0,1)\setminus E$ and $p\in\mathbb{P}_m$ is such that $s(\mathcal{F}_\lam,p)>1$, then $\mu(\mathcal{F}_\lambda,p)$ is an absolutely continuous measure.
\end{thm}
Of course, Theorem \ref{thm:BC-biased} is simply the above with $m=2$, $a_1=-1, a_2=1$. Unfortunately, the proof does not seem to give any information about the densities. On the other hand, it is known that absolutely continuous self-similar measures are equivalent to Lebesgue measure on their support \cite{MauldinSimon98,PSS00}. Also, this method does not provide any new \emph{explicit} examples of absolutely continuous self-similar measures.

In addition to improving existing results on the dimension of exceptions (and giving the first absolute continuity results in part of the parameter region $p\notin [1/3,2/3]$), there is a noteworthy difference between our proof and the earlier results on absolute continuity \cite{Solomyak95,PeresSolomyak96, PeresSolomyak98, PeresSchlag00,LPS02}. Namely, in all those papers a.e. absolute continuity is first established in certain interval $(1/2,\lambda_*)$ (where $\lambda_*>2^{-2/3}$) by using the ``transversality'' method. The result is then extended to the whole interval $(1/2,1)$ by using the ``thinning and convolving'' technique. See the survey \cite{Soloymak04} for an exposition of these ideas. By contrast, in our approach transversality does not come in (the result of Hochman that we rely on uses a weak ``higher order'' version of transversality, but it is so weak that it is automatically satisfied on the whole of $(0,1)$).

In a forthcoming joint article with B. Solomyak we obtain versions of Theorem \ref{thm:homogeneous} in which the parameter comes in the translations rather than the contraction ratio. As an application, we are able to show that the orthogonal projections of homogeneous self-similar measures without rotations on $\R^2$, of dimension $>1$, are absolutely continuous, off a dimension zero set of angles.

\subsection{Convolutions of Cantor measures}

Our second main result concerns convolutions of self-similar measures. Let $A_\lambda$ be the support of $\nu_\lambda$ which, for $\lambda\in (0,1/2)$ is a self-similar Cantor set of dimension $\log 2/|\log\lambda|$. For sets $A,B\subset\R$, their arithmetic sum is $A+B=\{a+b:a\in A,b\in B\}$. For a fixed compact set $K\subset \R$, Peres and Solomyak \cite{PeresSolomyak98} proved that for almost all $\lambda\in (0,1/2)$, $A_\lambda+K$ has Hausdorff dimension $\log 2/|\log\lambda|+\dim_H(K)$ if this number is $\le 1$, and has positive measure otherwise. In their proof they in fact establish an analogous result for measures (with convolution in place of arithmetic sum). Peres and Shmerkin \cite{PeresShmerkin09} proved that, when $K$ is a scaled copy of $A_{\lambda'}$, there are only countably many exceptions for the dimension result (uniformly in the scaling):
\[
\dim_H(A_{\lambda}+ r A_{\lambda'}) = \min(\dim_H(A_\lambda)+\dim_H(A_{\lambda'}),1) \quad\text{if } \frac{\log\lambda}{\log\lambda'} \notin\mathbb{Q}, r\neq 0.
\]
A version of this result for measures was obtained in \cite{NPS12}; see Theorem \ref{thm:corr-dim-convolutions} below. In that paper, it was shown that one cannot expect a similar result in the super-critical regime. Namely, if $\lambda=1/3$, $\lambda'=1/4$, then there is a dense $G_\delta$ set of scalings $r$ such that $\nu_\lambda * (S_r \nu_{\lambda'})$ is a singular measure, where $S_r(x)=rx$ (note that the sum of the similarity dimensions is $>1$, hence we are in the super-critical case). The result holds whenever $\lambda,\lambda'$ are reciprocals of Pisot numbers and $\log\lambda/\log\lambda'\notin\mathbb{Q}$ \cite[Theorem 4.1]{NPS12},  thus giving an at least countable set of exceptions to the statement that $\nu_\lambda * (S_r \nu_{\lambda'})$ is absolutely continuous for all $r\in\mathbb{R}\setminus\{0\}$. We show that the exceptional set is zero-dimensional, in a more general context. We denote by $\text{HOM}_\lam$ the family of iterated function systems of the form $(\lam x+a_i)_{i=1}^m$. The \emdef{uniform} self-similar measure for $\mathcal{F}$ is $\mu(\mathcal{F},p)$ for $p=(1/m,\ldots,1/m)$.

\begin{thm} \label{thm:convolution-ssm}
There exists a set $E\subset (0,1)$ of zero Hausdorff dimension such that the following holds. Let $\lam_1,\lam_2\in (0,1)$ such that $\lam_1\notin E$ and $\log \lam_2/\log \lam_1\notin\mathbb{Q}$.

Let $\mathcal{F}_i\in\text{HOM}_{\lam_i}$ satisfy the SSC and suppose that $s(\mathcal{F}_1)+s(\mathcal{F}_2)>1$. Then $\mu_1 * (S_r\mu_2)$ is absolutely continuous for all $r\neq 0$, with fractional derivatives in $L^2$, where $\mu_i$ is the uniform self-similar measure generated by $\mathcal{F}_i$.

In particular, $A(\mathcal{F}_1)+r A(\mathcal{F}_2)$ has positive Lebesgue measure.
\end{thm}

We make some remarks on this statement.
\begin{enumerate}
\item Unlike Theorem \ref{thm:homogeneous}, we get some information about the densities here. This is because the results in \cite{NPS12} that we rely on are for correlation dimension rather than Hausdorff dimension.
\item Note that the exceptional set depends only on one of the contraction ratios (and the standard assumption that the ratios are incommensurable; this is necessary in general). It does not depend on the translations.
\item The family $\{A(\mathcal{F}_1)+t A(\mathcal{F}_2)\}_{t\in\R}$ is a reparametrization of the orthogonal projections of the product set $A(\mathcal{F}_1)\times A(\mathcal{F}_2$) (other than the vertical projection). Thus the theorem says that under the stated conditions on $\lam_1,\lam_2$, the only exceptions for Marstrand's Theorem (see e.g. \cite[Theorem 6.1]{Falconer03}) for these product sets are the principal directions.
\item There are analogs for self-similar measures with more general weights, and also for convolutions of more than two self-similar measures. These can be obtained by adapting the proof with help of the remarks in \cite[Section 5]{NPS12}. Details are left to the interested reader.
\item In connection with some deep problems related to homoclinic bifurcations, Palis \cite{Palis87} conjectured that ``generically'' the arithmetic sum of two Cantor sets in the real line either has dimension less than $1$, or has nonempty interior. This conjecture was verified in a variety of settings, notably by Moreira and Yoccoz \cite{MoreiraYoccoz01} in the context most relevant to Palis' motivation. Theorem \ref{thm:convolution-ssm} goes in a similar direction: although only positive Lebesgue measure (rather than nonempty interior) is established, the ``generic'' part is very strong, as the possible exceptions lie in a very small set.
\end{enumerate}

Many other variants of Theorem \ref{thm:convolution-ssm} are possible. We state only one, which uses \cite[Theorem 1.4]{HochmanShmerkin12} instead of the results of \cite{NPS12}. We say that an IFS $\mathcal{F}=(f_1,\ldots,f_m)$ on $\R$ is \emdef{regular} if each $f_i$ is a strictly increasing $C^{1+\e}$ map, and there is a nonempty open interval $I\subset\R$ such that $f_i(I)$ are mutually disjoint. A \emdef{Gibbs} measure on the attractor $A(\mathcal{F})$ is the projection of a Gibbs measure for a continuous potential on the symbol space $\{1,\ldots,m\}^\N$ under the projection $\pi_{\mathcal{F}}$. Given a regular IFS $\mathcal{F}$, we denote $\Lambda(\mathcal{F})= \{ \log g'_i(x_i)\}_{i=1}^m$, where $x_i$ is the fixed point of $g_i$.
\begin{thm}
Let $E$ be the exceptional set from (the proof of) Theorem \ref{thm:convolution-ssm}, and fix $\lambda\in (0,1)\setminus E$. Then for any IFS $\mathcal{F}\in\text{HOM}_\lam$, and any regular IFS $\mathcal{G}=(g_1,\ldots,g_n)$ such that $\log \lambda'/\log\lambda\notin\mathbb{Q}$ for some $\lambda'\in\Lambda(\mathcal{F})$, the following holds: for any self-similar measure $\mu$ for $\mathcal{F}$, and any Gibbs measure $\nu$ for $\mathcal{G}$ such that $\dim\mu+\dim\nu>1$, the convolution $\mu*\nu$ is absolutely continuous.

In particular, if $\dim_H(A(\mathcal{F}))+\dim_H(A(\mathcal{G}))>1$, then the arithmetic sum $A(\mathcal{F})+A(\mathcal{G})$ has positive Lebesgue measure.
\end{thm}
The proof is a minor variant of the proof of Theorem \ref{thm:convolution-ssm} (using \cite[Theorem 1.4]{HochmanShmerkin12}) and is left to the reader. As an example, we get:
\begin{cor}
If $D\subset\mathbb{N}$ is any finite set with at least two elements, and $B_D\subset [0,1]$ are the numbers whose continued fraction expansion has digits only in $D$, then $A_\lam+B_D$ has positive measure for all $\lam\in (0,1)\setminus E$ whenever $\dim_H(A_\lam)+\dim(B_D)>1$.
\end{cor}
This is immediate from the fact that $B_D$ is the invariant set for a regular IFS which automatically satisfies the algebraic assumption (see e.g. \cite[Proof of Theorem 1.12]{HochmanShmerkin13}). 

\section{Background and preliminary results}

\subsection{Dimensions of measures}

We have already met the lower Hausdorff dimension of a measure $\mu$, $\dim\mu$. It is well-known that $\dim\mu$ can also be expressed in terms of local dimensions of $\mu$:
\begin{equation} \label{eq:local-dim}
\dim\mu = \text{essinf}_{x\sim \mu} \liminf_{r\searrow 0} \frac{\log\mu(B(x,r))}{\log r}.
\end{equation}
This is a version of the mass distribution principle, see e.g. \cite[Proposition 2.3]{Falconer97}. We will also need to make use of the \emdef{(lower) correlation dimension} of a measure:
\[
\dim_2\mu = \liminf_{r\searrow 0} \frac{\log \int \mu(B(x,r)) d\mu(x)}{\log r}.
\]
It holds that $\dim_2\mu\le \dim\mu$ for any measure $\mu$, with strict inequality possible.

Recall that the $s$-energy $I_s\mu$ of $\mu\in\PP$ is given by
\[
I_s\mu = \int\int \frac{d\mu(x)d\mu(y)}{|x-y|^s}.
\]
It is well-known (and easy to check) that $\dim_2\mu= \sup\{s\ge 0: I_s\mu<+\infty \}$. On the other hand, the energy can be expressed in terms of the Fourier transform of $\mu$, namely there is a constant $c(s)>0$ such that
\begin{equation} \label{eq:energy-Fourier}
I_s\mu = c(s)\int |\xi|^{s-1}|\widehat{\mu}(\xi)|^2 d\xi.
\end{equation}
See e.g. \cite[Lemma 12.12]{Mattila95}. Thus, if $s<\dim_2\mu$, then the Fourier transform of $\mu$ decays like $|\xi|^{-s/2}$ in average.

\subsection{Measures with power Fourier decay}

Let $\mathcal{D}$ be the class of measures whose Fourier transform has at least power decay at infinity:
\[
\mathcal{D} = \{\mu\in\mathcal{P}: |\widehat{\mu}(\xi)| \le C |\xi|^{-s} \text{ for some }C,s>0\}.
\]
This class will play a critical r\^{o}le, thanks to the following lemma.
\begin{lemma} \label{lem:convolutionAbsCont}
 Let $\nu\in\mathcal{D},\mu\in\mathcal{P}$.
\begin{enumerate}
\item If $\dim_2\mu=1$, then $\nu*\mu$ is absolutely continuous with a density in $L^2$, and even with fractional derivatives in $L^2$.
\item If $\dim\mu=1$, then $\nu*\mu$ is absolutely continuous.
\end{enumerate}
\end{lemma}
\begin{proof}
By assumption, there is $t>0$ such that $\widehat{\nu}(\xi) = O(|\xi|^{-t})$. Since $\dim_2\mu=1$, it follows from \eqref{eq:energy-Fourier} applied to $s=1-t/2$ that
\[
\int |\xi|^{-t/2} |\widehat{\mu}(\xi)|^2 d\xi < \infty.
\]
By the convolution formula,
\[
\int |\xi|^{t/2} |\widehat{\nu*\mu}(\xi)|^2 d\xi <\infty.
\]
Thus $\nu*\mu$ has fractional derivatives in $L^2$, giving the first assertion.

For the second statement, we note that, thanks to Egorov's Theorem and \eqref{eq:local-dim}, for every $\e>0$ there are $C_\e>0$ and a set $A_\e$ with $\mu(A_\e)>1-\e$ such that $\mu_\e:= \mu|_{A_\e}/\mu(A_\e)$ satisfies
\[
\mu_\e(B(x,r)) \le C_\e\, r^{1-s/4} \quad\text{for all }x\in A_\e.
\]
In particular, $\dim_2\mu_\e\ge 1-s/4$. The same argument as above then shows that $\nu*\mu_\e$ is absolutely continuous. Letting $\e\to 0$ finishes the proof.
\end{proof}

It is known since Erd\"{o}s \cite{Erdos40} and Kahane \cite{Kahane71} that Bernoulli convolutions are in $\mathcal{D}$, outside a zero-dimensional set of parameters (Erd\"{o}s proved this for a Lebesgue null set of parameters; Kahane observed the argument yields in fact dimension zero). This result is a corollary of a combinatorial fact that we state separately as it will have other pleasant consequences for us. Given a real number $x$, let $\|x\|$ denote its distance to the closest integer.

\begin{prop} \label{prop:combinatorial}
Let $G_\ell$ be the set of all real numbers $\theta>1$ such that
\begin{equation} \label{eq:combinatorial}
\liminf_{N\to\infty}\frac{1}{N}\min_{t\in [1,\theta]}\left|\left\{ n\in \{0,\ldots,N-1\} : \| t \theta^n\|\ge 1/\ell  \right\}\right| > 1/\ell.
\end{equation}
Then for any $1<\Theta_1<\Theta_2<\infty$ there is a $C=C(\Theta_1,\Theta_2)>0$ such that
\[
\dim_H([\Theta_1,\Theta_2]\setminus G_\ell)\le \frac{C\log(C\ell)}{\ell}.
\]
\end{prop}
The proof of this is contained in the proof of \cite[Proposition 6.1]{PSS00}. We obtain the following consequence. This was observed by T. Watanabe \cite[Theorem 1.5]{Watanabe12}, but we include the proof for the reader's convenience.

\begin{prop} \label{prop:FourierDecay}
There exists a set $E\subset (0,1)$ of Hausdorff dimension zero such that the following holds. Suppose $\lam\in (0,1)\setminus E$. Then for any IFS of the form $\mathcal{F}=( \lambda x+a_1,\ldots, \lambda x+a_m)$ with all the $a_i$ different, and any $p\in\mathbb{P}_m$, the self-similar measure $\mu(\mathcal{F},p)$ belongs to $\mathcal{D}$.
\end{prop}
\begin{proof}
Let
\[
E = \left\{\lam: \lam^{-1}\in  (1,\infty)\setminus \cup_{\ell\in \N} G_\ell\right\},
\]
where $G_\ell$ are the sets given by Proposition \ref{prop:combinatorial}. Then this proposition shows that $\dim_H(E)=0$.

Fix $\lambda\in (0,1)\setminus E$, distinct numbers $a_1,\ldots, a_m$ and $p\in\mathbb{P}_m$. By translating, scaling and relabeling (which does not affect the claim) we can, and do, assume that $a_1=0$ and $a_2=1$. Let $\mu$ be the corresponding self-similar measure. It is well-known that
\[
\widehat{\mu}(\xi) = \prod_{n=0}^\infty \Phi(\lambda^n \xi),
\]
where
\begin{equation} \label{eq:exponential-factor-in-FT}
\Phi(\zeta) = \sum_{j=1}^m p_j\, \exp(i \pi  a_j \zeta).
\end{equation}
(This follows easily either from the self-similarity or the fact that $\mu$ is an infinite convolution.) By assumption, there is $\ell\in\N$ such that \eqref{eq:combinatorial} holds with $\theta=\lam^{-1}$. We observe that there is $\delta>0$ (depending on $\ell$ and all the given data) such that $\left|\Phi(\zeta)\right|\le 1-\delta$ whenever $\|\zeta\|\ge 1/\ell$ (here we use our normalization $a_0=0; a_1=1$). Hence, if $\xi=t \lambda^{-N}$ with $t\in [1,\lambda^{-1}]$ and $N$ is large enough,
\[
|\widehat{\mu}(\xi)| \le \prod_{n=0}^{N-1} \left|\Phi(t\lambda^{-n})\right| \le (1-\delta)^{N/\ell} = O(|\xi|^{-s}),
\]
for $s=\tfrac{\log(1-\delta)}{\ell\log\lambda}>0$.
\end{proof}

The exceptional set $E$ from this proposition is closely connected to the exceptional set in our main theorems. Unfortunately, it appears that no explicit elements of $(0,1)\setminus E$ are known. Logarithmic decay has recently been established by Dai \cite[Proposition 2.5]{Dai12} and Bufetov and Solomyak \cite[Corollary 7.5]{BufetovSolomyak13} for some classes of algebraic $\lambda$, but this is not enough for our purposes. Dai, Feng and Wang \cite[Theorem 1.6]{DFW07} show that the Fourier transform of some self-similar measures has power decay, but it does not follow that the contraction ratios are in $(0,1)\setminus E$, nor does it lead to any new explicit examples of absolutely continuous self-similar measures. On the other hand, $E$ contains the reciprocals of Pisot numbers (since $\widehat{\nu}_\lam(\xi)\nrightarrow 0$ as $\xi\to\infty$ in this case) as well as reciprocals of Salem numbers (\cite{Kahane71}, see also \cite[Lemma5.2]{PSS00}); recall that an algebraic number $\theta>1$ is Salem if all of its algebraic conjugates lie on the closed unit disk, with at least one of them on the unit circle.

\section{Proofs of main results}

\subsection{Proof of Theorem \ref{thm:homogeneous}}

In order to prove Theorem \ref{thm:homogeneous}, we need the following result of Hochman. Recall that given an IFS $\mathcal{F}=(f_1,\ldots, f_m)$, the \emdef{projection map} $\pi=\pi_{\mathcal{F}}:\{1,\ldots, m\}^\N\to \R$ is given by
\[
\pi(\iii) = \lim_{n\to\infty} f_{i_1}\circ \cdots \circ f_{i_n}(0).
\]
The significance of this map is that self-similar set $A(\mathcal{F})$ is the image of $\pi$, and $\mu(\mathcal{F},p)$ is the push-down of the $p$-Bernoulli measure on $\{1,\ldots,m\}^\N$ under $\pi$.

\begin{thm}[\cite{Hochman13}, Theorem 1.8] \label{thm:zero-dim-excpetion-for-dim}
Let
\[
\big\{ \mathcal{F}_t = ( \lam_1(t)x+ a_1(t),\ldots, \lam_m(t)x+a_m(t) ) \big\}_{t\in I}
\]
be a one-parameter family of iterated function systems, where the maps $\lam_i:I\to (-1,1)\setminus\{0\}$ and $a_i:I\to\R$ are real analytic, and the following non-degeneracy condition holds:  for all distinct $\iii,\jjj\in\{1,\ldots,m\}^\N$, there is $t\in I$ such that $\pi_{\mathcal{F}_t}(\iii)\neq \pi_{\mathcal{F}_t}(\jjj)$.

Then there exists a set $E\subset I$ of zero Hausdorff (and even packing) dimension, such that if $t\in I\setminus E$ and $p\in\mathbb{P}_m$, then
\[
\dim\mu(\mathcal{F}_t,p) = \min(s(\mathcal{F}_t,p),1).
\]
\end{thm}

\begin{proof}[Proof of Theorem \ref{thm:homogeneous}]
Fix $k\in \N$. Consider the IFS
\[
\mathcal{F}_{\lam}^{(k)} = \left( \lambda^k x+ \sum_{j=0}^{k-2} a_{i_{j+1}} \lambda^j \right)_{\iii\in \{1,\ldots,m\}^{k-1}}.
\]
Further, if $p\in\mathbb{P}_m$, we write
\begin{equation} \label{eq:def-p-iterate}
p^{(k)}=(p_{i_1}\cdots p_{i_{k-1}})_{\iii\in \{1,\ldots,m\}^{k-1}}.
\end{equation}
The weighted IFS ($\mathcal{F}_\lam^{(k)}, p^{(k)})$ corresponds to ``skipping every $k$-th digit of $(\mathcal{F}_\lam,p)$''. We make some simple observations:
\begin{enumerate}
\item \label{it:sim-dim-skip} For any $p\in\mathbb{P}_m$,
\[
s(\mathcal{F}_{\lam}^{(k)},p^{(k)})=\left(1-\frac{1}{k}\right)s(\mathcal{F}_\lam,p).
\]
\item \label{it:non-degenerate} The family $\{ \mathcal{F}_\lam^{(k)}\}_{\lam\in (0,1)}$ satisfies the non-degeneracy assumption in Theorem \ref{thm:zero-dim-excpetion-for-dim}. Indeed, if $\iii\neq\jjj$, then $\pi_{\mathcal{F}_\lam^{(k)}}(\iii)-\pi_{\mathcal{F}_\lam^{(k)}}(\jjj)$ is a non-trivial power series in $\lambda$ with bounded coefficients (because the $a_i$ are all distinct)
\item \label{it:convolution} $\mu(\mathcal{F}_\lam,p)= \mu(\mathcal{F}_{\lam^k},p)*\mu(\mathcal{F}_\lam^{(k)},p^{(k)})$. This is immediate either from the definition, or from realizing $\mu(\mathcal{F}_\lam,p)$ as a convolution of discrete measures $(\mu_n)_{n\ge 1}$ and then splitting the values of $n$ such that $k\nmid n$ (which yields $\mu(\mathcal{F}_\lam^{(k)},p^{(k)})$) and the values of $n$ such that $k\mid n$ (which yields  $\mu(\mathcal{F}_{\lam^k},p)$).
\end{enumerate}

Now \eqref{it:sim-dim-skip}, \eqref{it:non-degenerate} and Theorem \ref{thm:zero-dim-excpetion-for-dim} imply that there exists a set $E_k$ of Hausdorff dimension zero such that
\[
\dim\mu(\mathcal{F}_{\lam}^{(k)},p^{(k)}) =1 \quad\text{ if }\lam\in (0,1)\setminus E_k \text{ and } s(\mathcal{F}_\lam,p)>\frac{k}{k-1}.
\]
Let $E'_k=\{ \lambda:\lambda^k\in \widetilde{E} \}$, where $\widetilde{E}$ is the exceptional set from Proposition \ref{prop:FourierDecay}. Then $\dim_H(E'_k)=0$. Moreover, from \eqref{it:convolution} and Lemma \ref{lem:convolutionAbsCont}, we deduce that if $\lam\in (0,1)\setminus (E_k\cup E'_k)$, and $s(\mathcal{F}_\lam,p)>1+1/k$, then $\mu(\mathcal{F}_\lam,p)$ is absolutely continuous. This yields the claim, with exceptional set $E= \bigcup_{k=1}^\infty (E_k\cup E'_k)$.
\end{proof}

\subsection{Proof of Theorem \ref{thm:convolution-ssm}}

For the proof of Theorem \ref{thm:convolution-ssm}, we appeal to the following result from \cite{NPS12}.
\begin{thm} \label{thm:corr-dim-convolutions}
Let $\mathcal{F}_i\in\text{HOM}_{\lambda_i}$ satisfy the SSC, $i=1,2$, and let $\mu_i$ be the corresponding uniform self-similar measures. If $\log \lam_2/\log\lam_1\notin\mathbb{Q}$, then
\[
\dim_2(\mu_1 *S_r\mu_2) = \min\left(s(\mathcal{F}_1)+s(\mathcal{F}_2),1\right)\quad\text{for all }r\neq 0.
\]
\end{thm}
This was proved in \cite[Theorem 1.1]{NPS12} for the convolutions $\nu_{\lam_1} * (S_r\nu_{\lam_2})$. However, as remarked in \cite[p.113]{NPS12}, the proof extends to this generality with very minor changes.

\begin{proof}[Proof of Theorem \ref{thm:convolution-ssm}]
Let $\widetilde{E}$ be the exceptional set from Proposition \ref{prop:FourierDecay}, and set
\[
E=\{\lambda:\lambda^k\in \widetilde{E} \text{ for some }k\in\N \}.
\]
Let $\mathcal{F}_1,\mathcal{F}_2$ be as in the statement of the theorem, and suppose $\lam_1\in E$ and $\log\lam_1/\log\lam_2\notin\mathbb{Q}$. As in the proof of Theorem \ref{thm:homogeneous}, given $k\in\N$ we may find two IFS's $\mathcal{F}_1^{(k)},\mathcal{G}^{(k)}\in\text{HOM}_{\lam_1^k}$ with uniform self-similar measures $\mu_1^{(k)},\nu^{(k)}$, such that $\mathcal{F}_1^{(k)}$ satisfies the SSC and has similarity dimension $(1-1/k)s(\mathcal{F}_1)$, and $\mu_1 = \nu^{(k)}*\mu_1^{(k)}$.

By the definition of $E$, $\nu^{(k)}\in\mathcal{D}$. On the other hand, it follows from Theorem \ref{thm:corr-dim-convolutions} that
\[
\dim_2\left(\mu_1^{(k)}*S_r\mu_2\right) = \min\left((1-1/k)s(\mathcal{F}_1)+s(\mathcal{F}_2),1\right) = 1 \quad\text{for all }r\neq 0,
\]
provided $k$ is taken large enough. Since $\mu_1* S_r\mu_2 =  \nu^{(k)}*(\mu_1^{(k)}*S_r\mu_2)$,  we only need to apply Lemma \ref{lem:convolutionAbsCont} to finish the proof.
\end{proof}

\bigskip
\textbf{Acknowledgment}. I thank Boris Solomyak and Michael Hochman for useful comments on an early version of this note, as well as many inspiring conversations on related topics over the years.


\end{document}